\documentclass[12pt]{amsart}

\usepackage{amsmath, amssymb, amscd, newlfont, bbm}
\usepackage[foot]{amsaddr}
\usepackage[final]{hyperref}
\hypersetup{
	colorlinks=true,        
	linkcolor={black},  
	citecolor={black},
	urlcolor={black}     
}

\newcommand{\Met}{{\mathrm{SL}_2(\mathbb R)^{\sim}}}
\newcommand{\SMet}{{\mathrm{SU}(1,1)^{\sim}}}
\newcommand{\spacedcdot}{{\,\cdot\,}}

\newcommand{\bbC}{{\mathbb{C}}}

\newcommand{\bbR}{{\mathbb{R}}}
\newcommand{\bbZ}{{\mathbb{Z}}}

\newcommand{\bfGamma}{{\mathbf{\Gamma}}}

\newcommand{\Lie}{{\mathrm{Lie}}}

\newcommand{\repspan}{{\mathrm{span}}}

\newcommand{\M}{{\mathrm{M}}}
\newcommand{\C}{{\mathrm{C}}}
\newcommand{\GL}{{\mathrm{GL}}}
\newcommand{\SU}{{\mathrm{SU}}}

\newcommand{\SO}{{\mathrm{SO}}}

\DeclareMathOperator{\Hol}{{\mathrm{Hol}}}

\DeclareMathOperator{\supp}{{\mathrm{supp}}}

\newcommand{\calC}{{\mathcal{C}}}

\newcommand{\frakg}{{\mathfrak{g}}}

\newcommand{\fraksl}{{\mathfrak{sl}}}

\usepackage[mathscr]{eucal}
\def \ScptA{\mathcal A}

\def \ScptH{\mathcal H}
\def \ScptD{\mathcal D}

\def\GL#1{{\mathrm{GL}}_{#1}}
\def\SL#1{{\mathrm{SL}}_{#1}}

\providecommand{\abs}[1]{\left\lvert#1\right\rvert}
\providecommand{\norm}[1]{\left\lVert#1\right\rVert}
\providecommand{\scal}[2]{\left<#1,#2\right>}
\providecommand{\mymod}[1]{\ (\mathrm{mod}\ #1)}

\numberwithin{equation}{section}

\newtheorem{Prop}[equation]{Proposition}
\newtheorem{Lem}[equation]{Lemma}
\newtheorem{Def}[equation]{Definition}
\newtheorem{Thm}[equation] {Theorem}
\newtheorem{Cor}[equation]{Corollary}

\title
   [ Poincar\' e Series on the Metaplectic Group]
   { On the Non-Vanishing of Poincar\' e Series on the Metaplectic Group}
\author{Sonja \v Zunar}

\address{ Department of Mathematics,
Faculty of Science,
University of Zagreb,
Bijeni\v cka 30,
10000 Zagreb,
Croatia}
\email{szunar@math.hr}

\thanks{The  author acknowledges Croatian Science Foundation grant no. 9364.}

\begin{document}
\maketitle

\begin{abstract}
	In this paper, we study the $ K $-finite matrix coefficients of integrable representations of the metaplectic cover of $ \SL2(\bbR) $ and give a result on the non-vanishing of their Poincar\' e series. We do this by adapting the techniques developed for $ \SL2(\bbR) $ by Mui\' c to the case of the metaplectic group.
\end{abstract}

\section{Introduction}

Let $ G $ be a connected semisimple Lie group with finite center, $ K $ a maximal compact subgroup of $ G $, and $ \Gamma $ a discrete subgroup of $ G $. The Poincar\' e series, $ \sum_{\gamma\in\Gamma}\varphi(\gamma g) $, of a function $ \varphi\in L^1(G) $ defines an element $ P_\Gamma\varphi $ of $ L^1(\Gamma\backslash G) $ \cite[\S4]{MuicMathAnn}. Let $ \pi $ be an integrable representation of $ G $. D.~Mili\v ci\' c observed that the Poincar\' e series of the $ K $-finite matrix coefficients of $ \pi $ span the isotypic component of $ \pi $ in the representation of $ G $ by right translations in $ L^2(\Gamma\backslash G) $ \cite[Lemma 6-6]{MuicSmooth}. In \cite{MuicJNT}, G.~Mui\' c studied these series and their non-vanishing in the case when $ G=\SL2(\bbR) $. In this paper, we adapt the techniques of \cite{MuicJNT} to the case when $ G $ is the metaplectic cover of $ \SL2(\bbR) $, which we call briefly the metaplectic group. 

We begin by presenting, in Section \ref{sec:002}, two realizations of the metaplectic group. The first is the classical one, $ \Met $, defined as the set of pairs $ (g,\eta) $, where $ g=\begin{pmatrix}a&b\\c&d\end{pmatrix}\in\SL2(\bbR) $, and $ \eta $ is a holomorphic function on the upper half-plane $ \ScptH $ such that $ \eta^2(z)=cz+d $ for all $ z\in\ScptH $. Multiplication in $ \Met $ is defined by \eqref{eq:007}. We fix its maximal compact subgroup $ K:=\left\{(g,\eta)\in\Met\right.:\left.g\in\SO_2(\bbR)\right\} $, whose unitary dual consists of characters $ \chi_n $, $ n\in\frac12\bbZ $, defined by \eqref{eq:053}. Our second realization of the metaplectic group, $ \SU(1,1)^\sim $, is constructed essentially by conjugating $ \Met $ via the Cayley transform $ \ScptH\to\ScptD $, $ z\mapsto\frac{z-i}{z+i} $, where $ \mathcal D $ is the unit disk.  

In Section \ref{sec:003}, we construct realizations of the genuine (anti)holomorphic discrete series of $ \Met $ on spaces of (anti)holomorphic functions on $ \ScptH $ (and on $ \ScptD $). This construction generalizes the classical construction of the (anti)ho\-lo\-mor\-phic discrete series of $ \SL2(\bbR) $ described in \cite[IX, \S2 and \S3]{lang}. It provides, for every $ m\in\frac32+\bbZ_{\geq0} $, a realization $ \pi_m $ (resp., $ \overline{\pi_m} $) of the unique irreducible unitary representation of $ \Met $ that decomposes, as a representation of $ K $, into the orthogonal sum $ \bigoplus_{k\in\bbZ_{\geq0}}\chi_{-m-2k} $ (resp., $ \bigoplus_{k\in\bbZ_{\geq0}}\chi_{m+2k} $). 

Let $ m\in\frac32+\bbZ_{\geq0} $ and $ k\in\bbZ_{\geq0} $. In Section 4, we find explicit formulae for $ K $-finite matrix coefficients $ F_{k,m} $ of $ \overline{\pi_m} $ that transform on the right as $ \chi_{m} $ and on the left as $ \chi_{m+2k} $. We do this analogously to the computation of matrix coefficients in \cite{MuicJNT}, by using the following observation. The classical lift of cuspidal modular forms of (half-integral) weight $ m $ to automorphic forms on $ \Met $ (see \cite[\S3.1]{gelbart}) extends naturally to the lift, defined by \eqref{eq:017}, of holomorphic functions $ \ScptH\to\bbC $ to smooth functions $ \Met\to\bbC $. By applying this lift to the elements of the $ \chi_{-m-2k} $-isotypic component of $ \pi_m $, one obtains functions on $ \Met $ whose immediate properties, combined with some Harish-Chandra's results, reveal them to be $ K $-finite matrix coefficients of $ \overline{\pi_m} $ that transform on the right as $ \chi_m $ and on the left as $ \chi_{m+2k} $. The formulae in Iwasawa coordinates for these matrix coefficients follow immediately, and we obtain those in Cartan coordinates in Lemma \ref{lem:036}. From this, it is possible to obtain (rather complicated) formulae, in Iwasawa coordinates, for all $ K $-finite matrix coefficients of $ \overline{\pi_m} $ (resp., of $ {\pi_m} $) that transform on both sides as characters of $ K $, as described in Corollary \ref{cor:042} and the paragraph following it.

In Section \ref{sec:005}, we recall some basic properties of Poincar\' e series on unimodular locally compact Hausdorff groups and give a short proof of the non-vanishing criterion \cite[Theorem 4-1]{MuicMathAnn}, relaxing the condition on the set $ C $ in \cite[Theorem 4-1]{MuicMathAnn} from compactness to measurability. This is Theorem \ref{thm:029}.

We begin Section \ref{sec:006} by giving a natural definition of square-integrable automorphic forms on $ \Met $. We show that, for $ m\in\frac52+\bbZ_{\geq0} $, the Poincar\' e series of functions $ F_{k,m} $ with respect to a discrete subgroup $ \Gamma $ of $ \Met $ converge absolutely and uniformly on compact sets and define bounded square-integrable automorphic forms on $ \Met $. 

Next, we study the non-vanishing of functions $ P_\Gamma F_{k,m} $, for $ k\in\bbZ_{\geq0} $ and $ m\in\frac52+\bbZ_{\geq0} $, in the case when $ \Gamma $ is a subgroup of $ \bfGamma(N):=\left\{(g,\eta)\in\Met:g\in\Gamma(N)\right\} $ for some $ N\in\bbZ_{>0} $, where $ \Gamma(N):=\left\{\begin{pmatrix}a&b\\c&d\end{pmatrix}\in\SL2(\bbZ):a,d\equiv1,\ b,c\equiv0\mymod N\right\}. $ In Lemma \ref{lem:044}, we show that $ P_\Gamma F_{k,m}=0 $ if $ \Gamma\cap K\neq\{1\} $. In the remaining case, we obtain, by applying Theorem \ref{thm:029}, a sufficient condition for the non-vanishing of functions $ P_\Gamma F_{k,m} $. This condition is essentially an inequality of integrals similar to the one in \cite[Lemma 6-5]{MuicJNT}. We interpret both of these inequalities in a new way, using the notion of the median of the beta distribution. Propositions \ref{prop:045} and \ref{prop:052} contain the final result. In the last part of the paper, we use properties of the median of the beta distribution and the power of mathematical software to provide two ready-to-use corollaries (Corollaries \ref{cor:054} and \ref{cor:055}).

\vspace{3mm}
This paper grew out of my PhD thesis. I would like to thank my advisor, Goran Mui\' c, for his encouragement, support, and many discussons. I would also like to thank Neven Grbac and Marcela Hanzer for their support and useful comments.

\section{The metaplectic group}\label{sec:002}

We start by establishing the basic notation.
Let $ \sqrt{\spacedcdot}:\bbC\to\bbC $ be the branch of the complex square root taking values in $\{z\in\bbC:\Re(z)>0\}\cup\{z\in\bbC:\Re(z)=0,\,\Im(z)\geq0\}$, and let $ i:=\sqrt{-1} $. We write $ z^m:=\left(\sqrt z\right)^{2m} $ for all $ m\in\frac12+\bbZ_{\geq0} $ and $ z\in\bbC $. Next, let 
$ \ScptH:=\left\{z\in\bbC:\Im(z)>0\right\} $, $ \ScptD:=\left\{z\in\bbC:\abs z<1\right\} $, 
and let $ \Hol(\ScptH) $ (resp., $ \Hol(\ScptD) $) denote the complex vector space of all holomorphic functions $ \ScptH\to\bbC $ (resp., $ \ScptD\to\bbC $).  

The group $ \GL2(\bbC) $ acts on $ \bbC\cup\{\infty\} $ by linear fractional transformations:
\begin{equation}\label{eq:004}
g.z:=\frac{az+b}{cz+d},\qquad g=\begin{pmatrix}a&b\\c&d\end{pmatrix}\in\GL2(\bbC),\ z\in\bbC\cup\{\infty\}.
\end{equation}
By restriction, $ \SL2(\bbR) $ acts on $ \ScptH $. 

The standard $ \SL2(\bbR) $-invariant Radon measure $ v_\ScptH $ on $ \ScptH $ is defined by the following formula (see \cite[\S1.4]{miyake}):
\[ \int_\ScptH f\,dv_\ScptH:=\int_0^\infty\int_\bbR f(x+iy)\,\frac{dx\,dy}{y^2},\qquad f\in C_c(\ScptH). \] 

Next, the function $ j:\GL2(\bbC)\times\bbC\to\bbC $,
\[ j(g,z):=cz+d,\qquad g=\begin{pmatrix}a&b\\c&d\end{pmatrix}\in\GL2(\bbC),\ z\in\bbC, \]
satisfies the following cocycle identity:
\begin{equation}\label{eq:003}
j\left(gg',z\right)=j\left(g,g'.z\right)j\left(g',z\right),\qquad g,g'\in\GL2(\bbC),\ z\in\bbC.
\end{equation}
We also note that
\begin{equation}\label{eq:010}
\Im(g.z)=\frac{\Im(z)}{\abs{j(g,z)}^2},\qquad g\in\SL2(\bbR),\ z\in\ScptH.
\end{equation}
For a proof of these facts, see \cite[\S1.1]{miyake}.

The metaplectic group is the unique (up to Lie group isomorphism) connected Lie group that is a covering group of degree 2 for $ \SL2(\bbR) $. We define it as the group
\[ \Met:=\left\{\sigma=\left(g_\sigma,\eta_\sigma\right)\in \SL2(\bbR)\times\Hol(\ScptH):\eta_\sigma^2(z)=j\left(g_\sigma,z\right)\text{ for all }z\in\ScptH\right\} \] 
with multiplication rule
\begin{equation}\label{eq:007}
\sigma_1\sigma_2=\bigl(g_{\sigma_1}g_{\sigma_2},\,\eta_{\sigma_1}(g_{\sigma_2}.z)\eta_{\sigma_2}(z)\bigr).
\end{equation}
We note that closedness of $ \Met $ under this multiplication is a consequence of \eqref{eq:003}. 

Before specifying a topology and smooth structure on $ \Met $, we note that any $ \sigma=(g_\sigma,\eta_\sigma)\in\Met $ is uniquely determined by the ordered pair $ \left(g_\sigma,\eta_\sigma(i)\right) $. This allows us to simplify the notation by identifying  $ \sigma \equiv \left(g_\sigma,\eta_\sigma(i)\right)  $. Now, we define a topology and smooth structure on $ \Met $ by requiring that the Iwasawa parametrization
\begin{equation}\label{eq:001}
(x,y,t)\mapsto\left(\begin{pmatrix}1&x\\0&1\end{pmatrix},1\right)\left(\begin{pmatrix}y^{\frac12}&0\\0&y^{-\frac12}\end{pmatrix},y^{-\frac14}\right)\left(\begin{pmatrix}\cos t&-\sin t\\\sin t&\cos t\end{pmatrix},e^{i\frac t2}\right)
\end{equation}
be a local diffeomorphism $ \bbR\times\bbR_{>0}\times\bbR\to\Met $. With this topology and smooth structure, $ \Met $ is a connected Lie group, and the projection
\[ \pi_1:\Met\to\SL2(\bbR),\qquad\pi_1(\sigma):=g_\sigma, \]
is a covering Lie group homomorphism of degree 2. 

\eqref{eq:004} motivates the following definition of a group action of $ \Met $ on $ \ScptH $:
\begin{equation}\label{eq:020a}
\sigma.z:=g_\sigma.z,\qquad \sigma\in\Met,\ z\in\ScptH.
\end{equation} 
Let us denote the three factors of the product on the right-hand side of \eqref{eq:001}, from left to right, by $ n_x $, $ a_y $, and $ \kappa_t $.
We have, for all $ x,y,t\in\bbR $ with $ y>0 $,
\begin{equation}\label{eq:018}
n_xa_y\kappa_t.i=x+iy\qquad\text{and}\qquad
\eta_{n_xa_y\kappa_t}(i)=y^{-\frac14}e^{i\frac t2}.
\end{equation}

Next, we put
\begin{equation}\label{eq:027}
h_t:=\left(\begin{pmatrix}e^t&0\\0&e^{-t}\end{pmatrix},e^{-\frac t2}\right),\qquad t\in\bbR_{\geq0}.
\end{equation}
The map 
\[ \bbR\times\bbR_{\geq0}\times\bbR\to\Met,\qquad \left(\theta_1,t,\theta_2\right)\mapsto\kappa_{\theta_1}h_t\kappa_{\theta_2}, \]
is a continuous surjection defining the Cartan coordinates of $ \Met $. Further, the unique Radon measure $ \mu_\Met $ on $ \Met $ satisfying
\begin{align}
\int_{\Met}f\,d\mu_{\Met}=&\frac1{4\pi}\int_0^{4\pi}\int_\ScptH f\left(n_xa_y\kappa_t\right)\,dv_{\ScptH}(x+iy)\,dt\label{eq:016}\\
=&\frac1{4\pi}\int_0^{4\pi}\int_0^\infty\int_0^{4\pi}f\left(\kappa_{\theta_1}h_t\kappa_{\theta_2}\right)\sinh(2t)\,d\theta_1\,dt\,d\theta_2\label{eq:035}
\end{align}
for all $ f\in C_c\left(\Met\right) $ is a Haar measure on the (connected semisimple, so) unimodular Lie group $ \Met $.

The stabilizer of $ i $ under the action \eqref{eq:020a} is $K:=\pi_1^{-1}\left(\SO_2(\bbR)\right)=\left\{\kappa_t:t\in\bbR\right\}$.
$ K $ is a maximal compact subgroup of $ \Met $, isomorphic to the Lie group $ (\bbR/4\pi\bbZ,+) $ via the map $ \kappa_t\mapsto t+4\pi\bbZ $. The unitary dual of $ K $ consists of characters $ \chi_n:K\to\bbC^\times $, $ n\in\frac12\bbZ $, defined by the formula
\begin{equation}\label{eq:053}
 \chi_n(\kappa_t):=e^{-int},\qquad t\in\bbR,\ n\in\frac12\bbZ. 
\end{equation}
We say that a function $ f:\Met\to\bbC $ transforms on the right (resp., on the left) as $ \chi_n $ if, for all $ \kappa\in K $ and $ \sigma\in\Met $, we have $ f(\sigma\kappa)=\chi_n(\kappa)f(\sigma) $ (resp., $ f(\kappa\sigma)=\chi_n(\kappa)f(\sigma) $). 

Let us identify the Lie algebra $ \mathfrak g:=\Lie\left(\Met\right) $ with $ \Lie\left(\SL2(\bbR)\right)\equiv\fraksl_2(\bbR) $ via the differential of $ \pi_1 $ at $ 1_\Met $. This identification of Lie algebras extends uniquely to an identification of universal enveloping algebras of their complexifications: $ U\left(\frakg_\bbC\right)\equiv U\left(\fraksl_2(\bbC)\right) $. 
In particular, matrices
\[ k^\circ:=\begin{pmatrix}0&-i\\i&0\end{pmatrix},\quad n^+:=\frac12\begin{pmatrix}1&i\\i&-1\end{pmatrix},\quad n^-:=\frac12\begin{pmatrix}1&-i\\-i&-1\end{pmatrix} \]
(cf.~\cite[p.~77]{ht}) form a standard basis of $ \frakg_{\bbC} $, i.e., they satisfy the following relations:
\begin{equation}\label{eq:058}
\left[n^+,n^-\right]=k^\circ,\qquad\left[k^\circ,n^+\right]=2n^+,\qquad\left[k^\circ,n^-\right]=-2n^-.
\end{equation}
Further, the Casimir element
\[ \mathcal C:=\frac12\left(k^\circ\right)^2-k^\circ+2n^+n^- \]
generates the center $ Z(\frakg_\bbC) $ of $ U(\frakg_\bbC) $. 

In any pair of local Iwasawa coordinates $ (x,y,t)\mapsto n_xa_y\kappa_t $ on $ \Met $ and $ (x,y,t)\mapsto \pi_1(n_xa_y\kappa_t) $ on $ \SL2(\bbR) $, the covering homomorphism $ \pi_1 $ is obviously given by a translation. Hence, in Iwasawa coordinates, the action of any left-invariant differential operator $ X\in U(\frakg_\bbC)\equiv U(\fraksl_2(\bbC)) $ on $ C^\infty\left(\Met\right) $ is given by the same formula as is its action on $ C^\infty\left(\SL2(\bbR)\right) $. In particular,  formulae (8) on page 116 and (1) on page 198 of \cite{lang} imply that
\begin{align}
	n^+&=iye^{-2it}\left(\frac\partial{\partial x}-i\frac\partial{\partial y}\right)+\frac i2e^{-2it}\frac{\partial}{\partial t},\label{eq:043}\\
	\mathcal C&=2y^2\left(\frac{\partial^2}{\partial x^2}+\frac{\partial^2}{\partial y^2}\right)+2y\frac{\partial^2}{\partial x\,\partial t}.\label{eq:013}
\end{align}

Let us define another realization, $ \SMet $, of the metaplectic group. We recall that the conjugation by $ {\displaystyle g:=\begin{pmatrix}1&-i\\1&i\end{pmatrix}} $ is an isomorphism of Lie groups $ \SL2(\bbR)\to\SU(1,1) $, while the linear fractional transformation corresponding to $ g $ defines an analytic isomorphism $ \ScptH\to\ScptD $. With this in mind, we put
\begin{align*}
&C=\left(g_C,\eta_C\right):= \left(\frac1{2i}\begin{pmatrix}1&-i\\1&i\end{pmatrix},\sqrt{\frac{z+i}{2i}} \right)\in\GL2(\bbC)\times\Hol\left(\ScptH\right),\\ 
&C^{-1}=\left(g_{C^{-1}},\eta_{C^{-1}}\right):= \left(\begin{pmatrix}i&i\\-1&1\end{pmatrix},\sqrt{1-w} \right)\in\GL2(\bbC)\times\Hol(\ScptD),	
\end{align*}
and define the $ C $-conjugate $ \sigma^C\in\SU(1,1)\times\Hol(\ScptD) $ of an element $ \sigma\in\Met $ by
\[ \sigma^C=\left(g_{\sigma}^C,\eta_{\sigma}^C\right):=\bigl(g_Cg_\sigma g_{C^{-1}},\eta_C\left(g_\sigma g_{C^{-1}}.w\right)\eta_\sigma\left(g_{C^{-1}}.w\right)\eta_{C^{-1}}(w)\bigr) \]
(cf.~\eqref{eq:007}).
Now, we define the Lie group $ \SMet $ by the condition that $ \spacedcdot^C $ be a Lie group isomorphism $ \Met\to\SMet $. We have
\[ \SMet=\left\{\sigma=(g_\sigma,\eta_\sigma)\in\SU(1,1)\times\Hol(\ScptD):\eta_\sigma(w)^2=j(g_\sigma,w)\text{ for all }w\in\ScptD\right\}, \]
multiplication law in $ \SU(1,1)^\sim $ is \eqref{eq:007}, and the projection onto the first coordinate is a smooth covering homomorphism $ \SMet \to \SU(1,1) $ of degree $ 2 $.

Now, let $ m\in\frac32+\bbZ_{\geq0} $. The formula
\begin{equation}\label{eq:008}
f\big|\left[\sigma\right]_m(z):=f(g_\sigma.z)\,\eta_\sigma(z)^{-2m}
\end{equation}
defines a right action of $ \Met $ on $ \Hol(\ScptH) $ and a right action of $ \SMet $ on $ \Hol(\ScptD) $. Further, by putting $ \sigma=C $ (resp., $ \sigma=C^{-1} $) in \eqref{eq:008}, we define a linear isomorphism $ \spacedcdot\big|\left[C\right]_m:\Hol(\ScptD)\to\Hol(\ScptH) $ and its inverse $ \spacedcdot\big|\left[C^{-1}\right]_m:\Hol(\ScptH)\to\Hol(\ScptD) $. 

\section{Genuine discrete series of the metaplectic group}\label{sec:003}

In this section, we present a construction of the genuine (anti)holomorphic discrete series of $ \Met $  that generalizes the classical construction of the (anti)holomorphic discrete series of $ \SL2(\bbR) $ described in \cite[IX, \S2 and \S3]{lang}.

For an admissible unitary representation $ (\pi,H) $ of $ \Met $, we denote by $ H_{\chi_n} $ the $ \chi_n $-isotypic component of $ \pi $, and by $ H_K $ the $ (\frakg,K) $-module of $ K $-finite vectors in $ H $.

Let $ m\in\frac32+\bbZ_{\geq0} $. We define the Hilbert space 
\[ H_{m}:=\left\{f\in\Hol(\ScptH):\int_\ScptH\abs{f(z)}^2\Im(z)^m\,dv_{\ScptH}(z)<\infty\right\} \]
with the following inner product: 
\[ \scal{f_1}{f_2}_{H_{m}}:=\int_\ScptH\left(f_1\overline{f_2}\right)\left(z\right)\Im(z)^m\,dv_\ScptH(z). \]
Next, we define a representation $ \left(\pi_{m},H_{m}\right) $ of $ \Met $ by the formula
\begin{equation}\label{eq:020}
\pi_{m}(\sigma)f:=f\big|\left[\sigma^{-1}\right]_m,\qquad\sigma\in\Met,\ f\in H_{m}. 
\end{equation}
This representation is unitary, which is easily checked using \eqref{eq:010} and the $ \SL2(\bbR) $-invariance of $ v_\ScptH $.
 
To make the $ K $-action in $ \left(\pi_{m},H_{m}\right) $ more obvious, we define a unitarily equivalent representation $ \left(\rho_{m},D_{m}\right) $ by requiring that the appropriate restriction of $ \spacedcdot\big|\left[C^{-1}\right]_m $ be a unitary equivalence $ H_{m}\to D_{m} $. One sees easily that
\[ D_{m}=\left\{f\in\Hol(\ScptD):\int_\ScptD\abs{f(w)}^2\left(1-\abs{w}^2\right)^{m-2}\,du\,dv<\infty\right\} \]
(we write $ w=u+iv $ with $ u,v\in\bbR $) and that the inner product on $ D_m $ is given by
\[ \scal{f_1}{f_2}_{D_{m}}:=4\int_\ScptD\left(f_1\overline{f_2}\right)\left(w\right)\left(1-\abs{w}^2\right)^{m-2}\,du\,dv,\qquad f_1,f_2\in D_{m}. \]
Further, we have
\begin{equation}\label{eq:011}
\rho_m(\sigma)f=f\Big|\left[\left(\sigma^C\right)^{-1}\right]_m,\qquad\sigma\in\Met,\ f\in D_{m}.
\end{equation}

To explore the $ K $-finite structure of $ \rho_{m} $, we note that, for all $ t\in\bbR $ and $ f\in D_{m} $,
\[ \rho_{m}(\kappa_t)f=f\Big|\left[\left(\kappa_t^C\right)^{-1}\right]_m=f\bigg|\left[\left(\begin{pmatrix}e^{it}&0\\0&e^{-it}\end{pmatrix},e^{-i\frac t2}\right)\right]_m=f(e^{2it}\spacedcdot)\,e^{imt}. \]
In particular, for every $ k\in\bbZ_{\geq0} $, the function $p_k:\ScptD\to\bbC$, $ p_k(w):=w^k $,
satisfies
\[ \rho_{m}(\kappa)p_k=\chi_{-m-2k}(\kappa)p_k,\qquad \kappa\in K. \]
Since $ \left\{p_k:k\in\bbZ_{\geq0}\right\} $ is an orthogonal basis of $ D_{m} $ (this can be shown by a proof identical to that of \cite[\S IX.3, Theorem 4]{lang}), it follows that $ \rho_{m} $ decomposes, as a representation of $ K $, into the orthogonal sum $ \bigoplus_{k\in\bbZ_{\geq0}}\chi_{-m-2k} $ and that, for any $ k\in\bbZ_{\geq0} $, $ (D_m)_{\chi_{-m-2k}} $ is spanned by $ p_k $. By transferring this result to $ H_{m} $ via the unitary equivalence given by a restriction of $ \spacedcdot\big|\left[C\right]_m $, we obtain the following lemma.

\begin{Lem}\label{lem:022}
	Let $ m\in\frac32+\bbZ_{\geq0} $. $ \pi_m $ decomposes, as a representation of $ K $, into the orthogonal sum $\bigoplus_{k\in\bbZ_{\geq0}}\chi_{-m-2k} $. 
	For any $ k\in\bbZ_{\geq0} $, $ (H_m)_{\chi_{-m-2k}} $ is spanned by the function
	\begin{equation}\label{eq:028}
		f_{k,m}:\ScptH\to\bbC,\qquad f_{k,m}(z):=(2i)^{m}\frac{(z-i)^k}{(z+i)^{m+k}}.
	\end{equation}
\end{Lem}

By replacing $\eta_\sigma(z)$ by $ \overline{\eta_\sigma(z)} $ in \eqref{eq:008} and holomorphicity by antiholomorphicity in all subsequent definitions, we obtain the definitions of two new admissible unitary representations of $ \Met $, which we will denote by $ \left(\overline{\pi_m},\overline{H_m}\right) $ and $ \left(\overline{\rho_m},\overline{D_m}\right) $, respectively. Computing as above, we see that, as $ K $-modules,
$ \left(\overline{H_m}\right)_K\cong\left(\overline{D_m}\right)_K\cong\bigoplus_{k\in\bbZ_{\geq0}}\chi_{m+2k} $
and that, for each $ k\in\bbZ_{\geq0} $, $ \left(\overline{D_m}\right)_{\chi_{m+2k}} $ (resp., $ \left(\overline{H_m}\right)_{\chi_{m+2k}} $) is spanned by $ \overline{p_k} $ (resp., by $ \overline{f_{k,m}} $).

\begin{Lem}\label{lem:011}
	Let $ m\in\frac32+\bbZ_{\geq0} $. $ \pi_{m} $ (resp., $ \overline{\pi_{m}} $), is the unique (up to unitary equivalence) irreducible unitary representation $ (\pi,H) $ of $ \Met $ such that $ H_K $ satisfies one of the following:
	\begin{enumerate}
		\item As a $ K $-module, $ H_K\cong\bigoplus_{k\in\bbZ_{\geq0}}\chi_{-m-2k} $ (resp., $ H_K\cong\bigoplus_{k\in\bbZ_{\geq0}}\chi_{m+2k} $).\label{lem:011:1}
		\item There exists $ v\in H_K\setminus\{0\} $ satisfying the following conditions:\label{lem:011:2}
		\begin{enumerate}
			\item $ \kappa.v=\chi_{-m}(\kappa)v,\ \kappa\in K $. (Resp., $ \kappa.v=\chi_{m}(\kappa)v,\ \kappa\in K $.)
			\item $ \mathcal C.v=m\left(\frac m2-1\right)v $.
		\end{enumerate}
	\end{enumerate}
\end{Lem}

\begin{proof}
	We sketch the proof for $ \left(\overline{\pi_m},\overline{H_m}\right) $. We have seen above that $ \left(\overline{H_m}\right)_K $ satisfies \eqref{lem:011:1}. 
	
	Next, we show that, if $ (\pi,H) $ is an admissible unitary representation of $ \Met $ such that $ H_K $ satisfies \eqref{lem:011:1}, then $ \mathcal C $ acts by $ m\left(\frac m2-1\right) $ on $ H_K $. Indeed, let $ v\in H_{\chi_m}\setminus\{0\} $. One sees easily that $ k^\circ.v=mv $ and, by using \eqref{eq:058}, that $ n^-.v\in H_{\chi_{m-2}}=0 $,
	so $ \mathcal C.v=m\left(\frac m2-1\right)v $. Since $ \mathcal C\in Z(\mathfrak g_\bbC) $, it follows that $ \mathcal C $ acts by $ m\left(\frac m2-1\right) $ on $ U(\mathfrak g_\bbC).v=:V $. It remains to show that $ V=H_K $. By the Poincar\' e-Birkhoff-Witt theorem, we have $ V=\sum_{j,s,t\in\bbZ_{\geq0}}\bbC\left(n^+\right)^j\left(k^\circ\right)^s\left(n^-\right)^t.v $. Since $ k^\circ.v=mv $ and $ n^-.v=0 $, this simplifies to $V=\sum_{j\in\bbZ_{\geq0}}\bbC\left(n^+\right)^j.v $. Further, \eqref{lem:011:1} and the classification of finite-dimensional $ \fraksl_2(\bbC) $-modules \cite[Proposition 1.1.12]{ht} make it impossible for $ H_K $ to have a non-zero finite-dimensional $ \fraksl_2(\bbC) $-submodule, so $ V $ is infinite-dimensional. Hence, for every $ j\in\bbZ_{\geq0} $, $ \left(n^+\right)^j.v\neq 0 $. Since $ \left(n^+\right)^j.v\in H_{\chi_{m+2j}} $, it follows, by \eqref{lem:011:1}, that $ V=H_K $.
	 
	Now, \cite[Theorem 1.3.1]{ht} gives the classification of $ \fraksl_2(\bbC) $-modules in which $ \mathcal C $ acts by a scalar and which decompose into a direct sum of finite-dimensional eigenspaces for $ k^\circ $. From it and the previous observation, it is clear that the condition \eqref{lem:011:1} uniquely determines the structure of $ H_K $ as an $ \fraksl_2(\bbC) $-module: it is the irreducible $ \fraksl_2(\bbC) $-module of lowest weight $ m $. In particular, $ \overline{\pi_m} $ is irreducible. 
	
	Thus, $ \overline{\pi_m} $ is, up to infinitesimal equivalence, the only irreducible unitary representation $ (\pi,H) $ of $ \Met $ such that $ H_K $ satisfies \eqref{lem:011:1} or, equivalently (again by \cite[Theorem 1.3.1]{ht}), \eqref{lem:011:2}. Since infinitesimal equivalence and unitary equivalence are the same for irreducible unitary representations of $ \Met $ by the well-known Harish-Chandra's result \cite[3.4.11]{W1}, this proves the lemma.
\end{proof}

\section{K-finite matrix coefficients of genuine square-integrable representations of the metaplectic group}\label{sec:004}

Throughout this section, let $ m\in\frac32+\bbZ_{\geq0} $. 

We recall the following definitions. The right regular representation $ \left(r,L^2\left(\Met\right)\right) $ is the unitary representation of $ \Met $ defined by the formula $ r(\sigma)f:=f(\spacedcdot\sigma) $. Similarly, the left regular represenation $ \left(l,L^2\left(\Met\right)\right) $ of $ \Met $ is defined by the formula $l(\sigma)f:=f\left(\sigma^{-1}\spacedcdot\right)$.

Let us find explicit formulae for $ K $-finite matrix coefficients of $ \overline{\pi_m} $ that transform on the right as $ \chi_m $ and on the left as $ \chi_{m+2k} $ for some $ k\in\bbZ_{\geq0} $. This could be done by a straightforward, but tedious, computation. We avoid it by using the following lemma.
 
\begin{Lem}\label{lem:005}
	Let $ k\in\bbZ_{\geq0} $. Suppose that $ F:\Met\to\bbC $ has the following properties:
	\begin{enumerate}
		\item $ F\in C^\infty\left(\Met\right)\cap L^2\left(\Met\right) $.\label{lem:005:1}
		\item $ \calC.F=m\left(\frac m2-1\right)F $.\label{lem:005:2}
		\item $ F $ transforms on the right as $ \chi_m $.\label{lem:005:3}
		\item $ F $ transforms on the left as $ \chi_{m+2k} $.\label{lem:005:4}
	\end{enumerate}
	Then, $ F $ is a $ K $-finite matrix coefficient of $ \overline{\pi_m} $.
\end{Lem}

\begin{proof}
	We prove the lemma by adapting the proof of \cite[Lemma 3-5]{MuicJNT} to the situation at hand, as follows.
	
	First, we note that the smallest closed subrepresentation $ H $ of $ \left(r,L^2\left(\Met\right)\right) $ containing $ F $ is an irreducible unitary representation of $ \Met $ unitarily equivalent to $ \overline{\pi_m} $. This follows from \eqref{lem:005:1}--\eqref{lem:005:3}, using \cite[Lemma 77 on page 89]{hc} and Lemma \ref{lem:011}.\eqref{lem:011:2}, as in the proof of \cite[Lemma 3-4]{MuicJNT}.
	
	Next, by \eqref{lem:005:1}, \eqref{lem:005:2}, and \eqref{lem:005:4}, using \cite[Theorem 1]{hc}, there exists $ \beta\in C_c^\infty\left(\Met\right) $ such that $ F=\overline{\beta\left(\spacedcdot^{-1}\right)}*F $. Now, the last three sentences of the proof of \cite[Lemma 3-5]{MuicJNT}, with $ F $ in place of $ F_{k,m} $, $ \Met $ in place of $ \SL2(\bbR) $, and $ K $ in place of $ K_\infty $, finish the proof of the lemma.
 \end{proof}	
	
To construct functions that satisfy the assumptions of Lemma \ref{lem:005}, we recall the classical lift of a function $ f\in\Hol(\ScptH) $ to the function $ F_f:\Met\to\bbC $,
\begin{equation}\label{eq:017}
F_f(\sigma):=f\big|[\sigma]_m(i)\overset{\eqref{eq:008}}=f(\sigma.i)\,\eta_\sigma(i)^{-2m},\qquad\sigma\in\Met
\end{equation}
(cf.~\cite[\S3.1]{gelbart}). In Iwasawa coordinates, we have, by \eqref{eq:018},
\begin{equation}\label{eq:015}
F_f(n_xa_y\kappa_t)=f(x+iy)y^{\frac m2}e^{-imt},\qquad x\in\bbR,\ y\in\bbR_{>0},\ t\in\bbR.
\end{equation}

\begin{Lem}\label{lem:014}
	Let $ f\in\Hol(\ScptH) $. We have the following:
	\begin{enumerate}
		\item $ F_f\in C^\infty\left(\Met\right) $.\label{lem:014:1}
		\item $ F_f $ transforms on the right as $ \chi_{m} $.\label{lem:014:2}
		\item $ \mathcal C.F_f=m\left(\frac m2-1\right)F_f $.\label{lem:014:3}
		\item $ \int_{\Met}\abs{F_f}^2\,d\mu_\Met=\int_\ScptH\abs{f(z)}^2\Im(z)^m\,dv_\ScptH(z) $.\label{lem:014:4}
	\end{enumerate}
\end{Lem}

\begin{proof}
	\eqref{lem:014:1} and \eqref{lem:014:2} are clear from \eqref{eq:015}. \eqref{lem:014:3} is obtained from \eqref{eq:013} and \eqref{eq:015} by an elementary computation, using that $ \left(\partial_x^2+\partial_y^2\right)f=0 $ since $ f $ is holomorphic. The proof of \eqref{lem:014:4} is an easy application of \eqref{eq:016} and \eqref{eq:015}. We leave the details to the reader.
\end{proof}

\begin{Lem}\label{lem:021}
	$ f\mapsto F_f $ defines an isometry $ \Phi:H_{m}\to L^2\left(\Met\right) $ that intertwines representations $ \left(\pi_{m},H_{m}\right) $ and $ \left(l,L^2\left(\Met\right)\right) $. In particular, for all $ k\in\bbZ_{\geq0} $,
	\begin{equation}\label{eq:024}
	F_{f_{k,m}}(\kappa\sigma)=\chi_{m+2k}(\kappa)F_{f_{k,m}}(\sigma),\qquad\kappa\in K,\ \sigma\in\Met.
	\end{equation}
\end{Lem}

\begin{proof}
	$ \Phi $ is a well-defined isometry by Lemma \ref{lem:014}.\eqref{lem:014:4}. It intertwines representations $ \pi_{m} $ and $ l $ by the following calculation:
	\[ \bigl(l\left(\tau\right)\Phi f\bigr)\left(\sigma\right)=F_f\left(\tau^{-1}\sigma\right)\overset{\eqref{eq:017}}=f\big|\left[\tau^{-1}\sigma\right]_m(i)\overset{\eqref{eq:020}}=\bigl(\pi_{m}(\tau)f\bigr)\big|\left[\sigma\right]_m(i)\overset{\eqref{eq:017}}=\Phi\bigl(\pi_{m}\left(\tau\right)f\bigr)\left(\sigma\right) \]
	for all $ \sigma,\tau\in\Met$ and $ f\in H_{m} $. In the case when $ \tau=\kappa^{-1} $ and $ f=f_{k,m}\in\left(H_{m}\right)_{\chi_{-m-2k}} $ (see Lemma \ref{lem:022}), this intertwining relation is \eqref{eq:024}.
\end{proof}

\begin{Prop}\label{prop:026}
	Let $ k\in\bbZ_{\geq0} $. Then, $ F_{k,m}:=F_{f_{k,m}} $ is a (unique up to a multiplicative constant) $ K $-finite matrix coefficient of $ \overline{\pi_{m}} $ that transforms on the right as $ \chi_m $ and on the left as $ \chi_{m+2k} $.
\end{Prop}

\begin{proof}
	First, we note that a $ K $-finite matrix coefficient $ F $ of $ \overline{\pi_m} $ that transforms on the right as $ \chi_m $ and on the left as $ \chi_{m+2k} $ is indeed unique up to a multiplicative constant because it is necessarily of the form $ \sigma\mapsto\scal{\overline{\pi_m}(\sigma)v_{m}}{v_{m+2k}} $, where each $ v_j $ is an element of the one-dimensional space $ \left(\overline{H_m}\right)_{\chi_j} $ (see Lemma \ref{lem:011}.\eqref{lem:011:1}). Next, to prove that $ F_{k,m} $ is such a matrix coefficient of $ \overline{\pi_m} $, we simply apply Lemma \ref{lem:005} to the function $ F_{k,m}=F_{f_{k,m}} $, which satisfies the assumptions \eqref{lem:005:1}--\eqref{lem:005:3} of Lemma \ref{lem:005} by Lemma \ref{lem:014} and satisfies the assumption \eqref{lem:005:4} of Lemma \ref{lem:005} by \eqref{eq:024}.
\end{proof}

\begin{Cor}\label{cor:042}
	Let $ j,k\in\bbZ_{\geq0} $. Then, the function $ \left(n^+\right)^j.F_{k,m} $ is a (unique up to a multiplicative constant) $ K $-finite matrix coefficient of $ \overline{\pi_m} $ that transforms on the right as $ \chi_{m+2j} $ and on the left as $ \chi_{m+2k} $. Its complex conjugate $\left(n^-\right)^j.\overline{F_{k,m}} $ is a (unique up to a multiplicative constant) $ K $-finite matrix coefficient of $ \pi_m $ that transforms on the right as $ \chi_{-m-2j} $ and on the left as $ \chi_{-m-2k} $.
\end{Cor}

\begin{proof}
	We have seen in the proof of Lemma \ref{lem:011} that, if $ v_m $ is a non-zero element of $\left(\overline{H_m}\right)_{\chi_m} $, then, for any $ j\in\bbZ_{\geq0} $, $ v_{m+2j}:=\overline{\pi_m}\left(n^+\right)^jv_m $ spans $ \left(\overline{H_m}\right)_{\chi_{m+2j}} $. Thus, the (unique up to a multiplicative constant) $ K $-finite matrix coefficient of $ \overline{\pi_m} $ that transforms on the right as $ \chi_{m+2j} $ and on the left as $ \chi_{m+2k} $ is given by
	\[ \sigma\mapsto\scal{\overline{\pi_m}(\sigma)\overline{\pi_m}\left(n^+\right)^jv_m}{v_{m+2k}}_{\overline{H_m}},\qquad\sigma\in\Met. \]
	By the definition of the derived representation and the continuity of $ \overline{\pi_m} $ and of the inner product, this is equal to
	\[ \sigma\mapsto\left(\left(n^+\right)^j.\scal{\overline{\pi_m}(\spacedcdot)v_m}{v_{m+2k}}_{\overline{H_m}}\right)(\sigma),\qquad\sigma\in\Met \]
	(here $ \left(n^+\right)^j $ acts as a left-invariant differential operator). Since, by Proposition \ref{prop:026}, there exists $ C\in\bbC^\times $ such that $ \scal{\overline{\pi_m}(\spacedcdot)v_m}{v_{m+2k}}=CF_{k,m} $, the claim for $ \overline{\pi_m} $ follows. The claim for $ \pi_m $ is its immediate consequence, since the definition of $ \pi_m $ is essentially the complex conjugate of the definition of $ \overline{\pi_m} $.
\end{proof}

One can obtain formulae in Iwasawa coordinates for all matrix coefficents of Corollary \ref{cor:042} using \eqref{eq:015}, \eqref{eq:028}, and \eqref{eq:043}. In Cartan coordinates, the functions $ F_{k,m} $ are given by the following formula:

\begin{Lem}\label{lem:036}
	We have, for all $ k\in\bbZ_{\geq0} $,
	\[ F_{k,m}(\kappa_{\theta_1}h_t\kappa_{\theta_2})=\chi_{m+2k}(\kappa_{\theta_1})\,\frac{\tanh^k(t)}{\cosh^m(t)}\,\chi_{m}(\kappa_{\theta_2}),\qquad \theta_1,\theta_2\in\bbR,\ t\in\bbR_{\geq0}. \]
\end{Lem}

\begin{proof}
	Since $ F_{k,m} $ transforms on the right as $ \chi_{m} $ and on the left as $ \chi_{m+2k} $ (by Proposition \ref{prop:026}), we have
	\[ F_{k,m}(\kappa_{\theta_1}h_t\kappa_{\theta_2})=\chi_{m+2k}(\kappa_{\theta_1})\,F_{k,m}(h_t)\,\chi_{m}(\kappa_{\theta_2}),\qquad \theta_1,\theta_2\in\bbR,\ t\in\bbR_{\geq0}. \]
	To finish the proof, we expand the middle factor on the right-hand side by its definition:
	\[ F_{k,m}(h_t)\overset{\eqref{eq:017}}=f_{k,m}(h_t.i)\,\eta_{h_t}(i)^{-2m}\overset{\eqref{eq:027}}=f_{k,m}\left(e^{2t}i\right)e^{mt}\overset{\eqref{eq:028}}=\frac{\tanh^k(t)}{\cosh^m(t)},\quad t\in\bbR_{\geq0}.\qedhere \]
\end{proof}

\begin{Lem}\label{lem:041}
	Let $ m\in\frac52+\bbZ_{\geq0} $, $ k\in\bbZ_{\geq0} $. Then, $ F_{k,m}\in L^1\left(\Met\right) $.
\end{Lem}

\begin{proof}
	We calculate, using Lemma \ref{lem:036} and  \eqref{eq:035},
	\begin{align*}
		\int_\Met\abs{F_{k,m}}\,d\mu_\Met&=\frac1{4\pi}\int_0^{4\pi}\int_0^\infty\int_0^{4\pi}\frac{\tanh^k(t)}{\cosh^m(t)}\sinh(2t)\,d\theta_1\,dt\,d\theta_2\\&\leq8\pi\int_0^\infty\frac{\sinh(t)}{\cosh^{m-1}(t)}\,dt=8\pi\int_1^\infty\frac{dx}{x^{m-1}}=\frac{8\pi}{m-2}<\infty,
	\end{align*}
	where the inequality is obtained by applying $ \tanh^k(t)\leq1 $ and $ \sinh(2t)=2\sinh(t)\cosh(t) $.
\end{proof}

By Corollary \ref{cor:042} and Lemma \ref{lem:041}, $ \pi_m $ and $ \overline{\pi_m} $ are integrable representations of $ \Met $ for every $ m\in\frac52+\bbZ_{\geq0} $.

\section{Preliminaries on Poincar\' e series}\label{sec:005}

Let $ G $ be a locally compact Hausdorff group that is secound-countable and unimodular, with a fixed Haar measure $ \mu_G $. Let $ \Gamma $ be a discrete subgroup of $ G $. We denote by $ \mu_{\Gamma\backslash G} $ the unique Radon measure on $ \Gamma\backslash G $ that satisfies
\[ \int_{\Gamma\backslash G}\sum_{\gamma\in\Gamma}f(\gamma g)\,d\mu_{\Gamma\backslash G}(g)=\int_Gf\,d\mu_G,\qquad f\in C_c(G). \]

For any $ \varphi\in L^1(G) $, the Poincar\' e series of $ \varphi $ with respect to $ \Gamma $ is defined by the formula
\begin{equation}\label{eq:034}
\left(P_\Gamma\varphi\right)(g):=\sum_{\gamma\in\Gamma}\varphi(\gamma g).
\end{equation}
This series converges absolutely almost everywhere on $ G $. Moreover, $ P_\Gamma\varphi\in L^1(\Gamma\backslash G) $ and 
\begin{equation}\label{eq:032}
\norm{P_\Gamma\varphi}_{L^1(\Gamma\backslash G)}\leq\norm\varphi_{L^1(G)}.
\end{equation}
For a proof of these facts, see the second paragraph of \cite[\S4]{MuicMathAnn}.

In Section \ref{sec:006}, we study the non-vanishing of Poincar\' e series of functions $ F_{k,m} $ using the criterion \cite[Theorem 4-1]{MuicMathAnn}. Here, we give a short proof of that criterion, at the same time relaxing the compactness condition on the set $ C $ in \cite[Theorem 4-1]{MuicMathAnn} to mere measurability. In other words, we have the following:

\begin{Thm}\label{thm:029}
	Let $ G $ be a locally compact Hausdorff group that is secound-countable and unimodular, with Haar measure $ \mu_G $. Let $ \Gamma $ be a discrete subgroup of $ G $. Let $ \varphi\in L^1(G) $. Suppose that there exists a Borel set $ C\subseteq G $ with the following properties:
	\begin{enumerate}
		\item[$ (\C1) $] $ \int_C\abs{\varphi}\,d\mu_G>\int_{G\setminus C}\abs\varphi\,d\mu_G $ 
		\item[$ (\C2) $] $ CC^{-1}\cap\Gamma=\{1_G\} $. 
	\end{enumerate}
	Then, $ P_\Gamma\varphi\neq 0 $ in $ L^1(\Gamma\backslash G) $.
\end{Thm}

\begin{proof}
	We denote by $ \mathbbm1_A $ the characteristic function of $ A\subseteq G $. $ (\C2) $ implies the following:
	\begin{equation}\label{eq:030}
	\text{for every }g\in G,\quad \#\{\gamma\in\Gamma:\mathbbm1_C(\gamma g)\neq0\}\leq1.
	\end{equation}
	Namely, if $ \gamma,\gamma'\in\Gamma $ and $ g\in G $ satisfy $ \mathbbm1_C(\gamma g)\neq 0 $ and $ \mathbbm1_C\left(\gamma' g\right)\neq0 $, i.e., if $ \gamma g,\gamma'g\in C $, then $ \gamma\gamma'^{-1}=(\gamma g)\left(\gamma'g\right)^{-1}\in CC^{-1}\cap\Gamma\overset{(\C2)}=\{1_G\} $, hence $ \gamma=\gamma'$.
	
	Now, we have the following:
	\begin{align}\label{eq:031}
		\begin{split}
		\norm{P_\Gamma(\varphi\cdot\mathbbm1_C)}_{L^1(\Gamma\backslash G)}&=\int_{\Gamma\backslash G}\Big\lvert\sum_{\gamma\in\Gamma}(\varphi\cdot\mathbbm1_C)(\gamma g)\Big\rvert\,d\mu_{\Gamma\backslash G}(g)\\&\overset{\eqref{eq:030}}=\int_{\Gamma\backslash G}\sum_{\gamma\in\Gamma}\abs{\left(\varphi\cdot\mathbbm1_C \right)(\gamma g)}\,d\mu_{\Gamma\backslash G}(g)\\&=\int_G\abs{\varphi\cdot\mathbbm1_C}\,d\mu_G\\&=\int_C\abs\varphi\,d\mu_G.
		\end{split}
	\end{align}
	On the other hand, by \eqref{eq:032},
	\begin{equation}\label{eq:033}
	\norm{P_\Gamma(\varphi\cdot\mathbbm1_{G\setminus C})}_{L^1(\Gamma\backslash G)}\leq\int_{G\setminus C}\abs\varphi\,d\mu_G.
	\end{equation}
	
	Now, we obtain
	\begin{align*}
		\norm{P_\Gamma\varphi}_{L^1(\Gamma\backslash G)}&\hspace{.4em}=\hspace{.4em}\norm{P_\Gamma(\varphi\cdot\mathbbm1_C)+P_\Gamma(\varphi\cdot\mathbbm1_{G\setminus C})}_{L^1(\Gamma\backslash G)}\\
		&\hspace{.4em}\geq\hspace{.4em}\norm{P_\Gamma(\varphi\cdot\mathbbm1_C)}_{L^1(\Gamma\backslash G)}-\norm{P_\Gamma(\varphi\cdot\mathbbm1_{G\setminus C})}_{L^1(\Gamma\backslash G)}\\
		&\underset{\eqref{eq:033}}{\overset{\eqref{eq:031}}\geq}\int_C\abs\varphi\,d\mu_G-\int_{G\setminus C}\abs\varphi\,d\mu_G\\
		&\overset{(\C1)}>0.
	\end{align*}
	Thus, $ P_\Gamma\varphi\neq0 $ in $ L^1(\Gamma\backslash G) $.
\end{proof}

\section{Poincar\' e series of functions $ F_{k,m} $}\label{sec:006}

\begin{Def}
	Let $ \Gamma $ be a discrete subgroup of $ \Met $.
	A square-integrable automorphic form on $ \Met $ with respect to $ \Gamma $ is a function $ F\in C^\infty\left(\Gamma\backslash\Met\right)\cap L^2\left(\Gamma\backslash\Met\right) $ such that the complex vector spaces $ \repspan_\bbC\left\{F\left(\spacedcdot\kappa\right):\kappa\in K\right\} $ and $ \left\{z.F:z\in Z(\frakg_\bbC)\right\} $ are finite-dimensional. The set of all such functions $ F $ will be denoted by $ \ScptA\left(\Gamma\backslash\Met\right) $.
\end{Def}

\begin{Lem}\label{lem:048}
	Let $ k\in\bbZ_{\geq0} $ and $ m\in\frac52+\bbZ_{\geq0} $. Let $ \Gamma $ be a discrete subgroup of $ \Met $. Then, the series $ \sum_{\gamma\in\Gamma}\abs{F_{k,m}(\gamma g)} $ converges uniformly on compact sets, and $ P_\Gamma F_{k,m}\in \mathcal A\left(\Gamma\backslash\Met\right) $. $ P_\Gamma F_{k,m} $ is bounded.
\end{Lem}

\begin{proof}
	$ F_{k,m} $ is integrable by Lemma \ref{lem:041} and satisfies the assumptions of Lemma \ref{lem:005} (see the end of the proof of Proposition \ref{prop:026}). Thus, by the argument on the first page of the proof of \cite[Theorem 3-10]{MuicMathAnn}, the series $ \sum_{\gamma\in\Gamma}\abs{F_{k,m}(\gamma g)} $ converges uniformly on compact sets, the function $ P_\Gamma F_{k,m} $ is in $ C^\infty\left(\Gamma\backslash\Met\right) $ and transforms on the right as $ \chi_m $, and $ \mathcal CP_\Gamma F_{k,m}=m\left(\frac m2-1\right)P_\Gamma F_{k,m} $. Further, by \cite[Theorem 1]{hc}, for any neighbourhood $ V $ of $ 1_\Met $, there exists $ \beta\in C_c^\infty\left(\Met\right) $ with $ \supp\beta\subseteq V $ such that $ F_{k,m}=\beta*F_{k,m} $, so $ P_\Gamma F_{k,m} $ is bounded by \cite[Lemma 6-3]{MuicSmooth}. Since it is also in $ L^1\left(\Gamma\backslash\Met\right) $, it is in $ L^2\left(\Gamma\backslash\Met\right) $.
\end{proof}

For $ N\in\bbZ_{>0} $, let us denote by $ \bfGamma(N) $ the preimage under $ \pi_1 $ of the congruence subgroup
\begin{equation}\label{eq:057}
\Gamma(N):=\left\{\begin{pmatrix}a&b\\c&d\end{pmatrix}\in\SL2(\bbZ):a,d\equiv1,\ b,c\equiv0\mymod N\right\}.
\end{equation}
Obviously, $ \bfGamma(N) $ is a discrete subgroup of $ \Met $. 

In the following, we use Theorem \ref{thm:029} and techniques of \cite{MuicJNT} to study the non-vanishing of $ P_\Gamma F_{k,m}\in\ScptA\left(\Gamma\backslash\Met\right) $, where $ k\in\bbZ_{\geq0} $, $ m\in\frac52+\bbZ_{\geq0} $, and $ \Gamma $ is a subgroup of $ \bfGamma(1) $. First, we prove that some of these series vanish, i.e., are identically $ 0 $ (cf.~the first claim of \cite[Lemma 6-5]{MuicJNT}).

\begin{Lem}\label{lem:044}
	Let $ \Gamma $ be a subgroup of $ \bfGamma(1) $ such that $ \Gamma\cap K\neq\{1\} $. Then, $ P_\Gamma F_{k,m}=0 $ for all $ k\in\bbZ_{\geq0} $ and $ m\in\frac52+\bbZ_{\geq0} $.
\end{Lem}

\begin{proof}
	Let $ k\in\bbZ_{\geq0} $ and $ m\in\frac52+\bbZ_{\geq0} $. We have, for all $ \sigma\in\Met $,
	\[ (P_\Gamma F_{k,m})(\sigma)\overset{\eqref{eq:034}}=\sum_{\gamma\in(\Gamma\cap K)\backslash \Gamma}\sum_{\delta\in\Gamma\cap K}F_{k,m}(\delta\gamma \sigma)\overset{\eqref{eq:024}}=\sum_{\gamma\in(\Gamma\cap K)\backslash\Gamma}\left(\sum_{\delta\in\Gamma\cap K}\chi_{m+2k}(\delta)\right)F_{k,m}(\gamma \sigma). \]
	To prove that this equals $ 0 $, it suffices to show that $ \sum_{\delta\in\Gamma\cap K}\chi_{m+2k}(\delta)=0 $. Since, by elementary algebra, the sum of all values of any non-trivial character of a finite Abelian group equals $ 0 $, it suffices to show that $ \chi_{m+2k}\big|_{\Gamma\cap K}\neq1 $. Indeed, since $ \Gamma\cap K $ is a non-trivial subgroup of the cyclic group $ \bfGamma(1)\cap K=\left\{\kappa_{n\frac\pi2}:n\in\bbZ\right\}\cong\bbZ/8\bbZ$, $ \Gamma\cap K $ contains the unique element of $ \bfGamma(1)\cap K $ of order 2, i.e., we have $ \kappa_{2\pi}\in\Gamma\cap K $, and $ \chi_{m+2k}(\kappa_{2\pi})=e^{-i(m+2k)2\pi}=-1 $. 
\end{proof}

We remark that a subgroup $ \Gamma $ of $ \bfGamma(1) $ satisfies the assumption $ \Gamma\cap K\neq\{1\} $ if and only if $ \Gamma $ contains  $\left\{1_\Met,\kappa_{2\pi}\right\}=\ker\pi_1 $, i.e., if and only if $ \Gamma $ is the preimage under $ \pi_1 $ of a subgroup of $ \SL2(\bbZ) $. This is obvious from the proof of Lemma \ref{lem:044}. 

Next, we study the remaining case, when $ \Gamma\cap K=\{1\} $. An important example of such a group $ \Gamma $ is the group $ \Gamma_1(4)^\sim $ defined as follows. For $ N\in\bbZ_{>0} $, let us denote
\[ \Gamma_1(N):=\left\{\begin{pmatrix}a&b\\c&d\end{pmatrix}\in\SL2(\bbZ):a,d\equiv1,\ c\equiv0\mymod N\right\}. \]
We define
\[ \Gamma_1(4)^\sim:=\{(\gamma,J(\gamma,z)):\gamma\in\Gamma_1(4)\}, \]
where $ J(\gamma, z) $ is the automorphic factor from the definition of modular forms of half-integral weight: $ J(\gamma,z):=\frac{\Theta(\gamma.z)}{\Theta(z)} $, $ \gamma\in\Gamma_1(4) $, $ z\in\ScptH $, where $ \Theta\in\Hol(\ScptH) $ is given by $ \Theta(z):=\sum_{n\in\bbZ}e^{2\pi in^2z} $. The map $ f\mapsto F_f $ defined by \eqref{eq:017} lifts cuspidal modular forms of half-integral weight $ m $ for $ \Gamma_1(4) $ to elements of $ \ScptA\left(\Gamma_1(4)^\sim\backslash\Met\right) $ (e.g., see \cite[\S3.1]{gelbart}).

To state our non-vanishing result, we use the following notion from probability theory.

\begin{Def}
	The median of the beta distribution $ B(a,b) $ with parameters $ a,b\in\bbR_{>0} $ is the unique $ \M(a,b)\in\left]0,1\right[ $ such that
	\begin{equation}\label{eq:047}
	\int_0^{\M(a,b)}x^{a-1}(1-x)^{b-1}\,dx=\int_{\M(a,b)}^1x^{a-1}(1-x)^{b-1}\,dx. 
	\end{equation}
\end{Def}

\begin{Prop}\label{prop:045}
	Let $ N\in\bbZ_{>0} $, $ k\in\bbZ_{\geq0} $, and $ m\in\frac52+\bbZ_{\geq0} $. Let $ \Gamma $ be a subgroup of $ \bfGamma(N) $ such that $ \Gamma\cap K=\{1\} $. If
	\begin{equation}\label{eq:039}
	N>\frac{4\,\M\left(\frac k2+1,\frac m2-1\right)^{\frac12}}{1-\M\left(\frac k2+1,\frac m2-1\right)}=:N_{k,m},
	\end{equation}
	then $ P_\Gamma F_{k,m} $ is not identically $ 0 $.  
\end{Prop}

\begin{proof}
	The plan is to apply to $ P_\Gamma F_{k,m} $ Theorem \ref{thm:029}, with $ C $ of the form
	\[ C_r:=K\,\{h_t:0\leq t\leq r\}\,K \]
	for some $ r\in\bbR_{>0} $ (cf.~the second claim of \cite[Lemma 6-5]{MuicJNT}). 
	
	First, the condition $ (\C1) $ for $ C=C_r $ and $ F=F_{k,m} $ is the following inequality:
	\[ \int_{C_r}\abs{F_{k,m}}\,d\mu_{\Met}>\int_{\Met\setminus C_r}\abs{F_{k,m}}\,d\mu_{\Met}. \]
	Using Lemma \ref{lem:036} and \eqref{eq:035}, this is equivalent to
	\[ \int_0^r\frac{\tanh^k(t)}{\cosh^{m}(t)}\,\sinh(2t)\,dt>\int_r^\infty\frac{\tanh^k(t)}{\cosh^m(t)}\,\sinh(2t)\,dt, \]
	which, after the substitution $ x=\tanh^2(t) $, becomes
	\[ \int_0^{\tanh^2(r)}x^{\frac k2}(1-x)^{\frac m2-2}\,dx>\int_{\tanh^2(r)}^1x^{\frac k2}(1-x)^{\frac m2-2}\,dx, \]
	which is obviously equivalent to
	\begin{equation}\label{eq:037}
	\tanh^2(r)>\M\left(\frac k2+1,\frac m2-1\right).
	\end{equation}
	
	Next, to explore the condition $ (\C2) $ for $ C=C_r $ with $ r\in\bbR_{>0} $, we define the function 
	\[ \norm\spacedcdot:\Met\to\bbR,\qquad \norm{\left(\begin{pmatrix}a&b\\c&d\end{pmatrix},\eta\right)}:=\sqrt{a^2+b^2+c^2+d^2}. \]
	It follows from \cite[Lemma 6-20]{MuicJNT} that
	\[ \norm\sigma\leq\sqrt{2\cosh(4r)},\qquad \sigma\in C_rC_r^{-1}. \]
	On the other hand, \cite[Lemma 6-21.(b)]{MuicJNT} implies that
	\[ \norm\sigma\geq\sqrt{N^2+2},\qquad\sigma\in\Gamma\setminus\pi_1^{-1}(\SO_2(\bbR))=\Gamma\setminus\{1\} \]
	(since $ \Gamma\cap K=\{1\} $). Hence, if $ r\in\bbR_{>0} $ satisfies
	\[ \sqrt{N^2+2}>\sqrt{2\cosh(4r)} \]
	or, equivalently, if
	\begin{equation}\label{eq:038}
	\tanh^2(r)<\left(\sqrt{\frac4{N^2}+1}-\frac2N\right)^2, 
	\end{equation}
	then $ C_rC_r^{-1}\cap\Gamma=\{1\} $, i.e., then $ C_r $ satisfies $ (\C2) $. 
	
	Thus, if some $ r\in\bbR_{>0} $ satisfies \eqref{eq:037} and \eqref{eq:038}, then $ C_r $ satisfies $ (\C1) $ and $ (\C2) $, from which it follows by Theorem \ref{thm:029} that $ P_\Gamma F_{k,m} $ is not identically zero. Now, an $ r\in\bbR_{>0} $ satisfying \eqref{eq:037} and \eqref{eq:038} obviously exists if and only if we have
	\[ \M\left(\frac k2+1,\frac m2-1\right)<\left(\sqrt{\frac4{N^2}+1}-\frac2N\right)^2, \]
	which is equivalent to \eqref{eq:039} by an elementary computation. This proves the proposition.
\end{proof}

An analogous computation enables us to rewrite the conditions of the second claim of \cite[Lemma 6-5]{MuicJNT} in terms of the median of the beta distribution:

\begin{Prop}\label{prop:052}
	Let $ m\in\bbZ_{\geq3} $ and $ k\in\bbZ_{\geq0} $. We define $ F_{k,m}\in C^\infty(\SL2(\bbR))\cap L^1(\SL2(\bbR)) $,
	\[ F_{k,m}\bigl(\pi_1(\kappa_{\theta_1}h_t\kappa_{\theta_2})\bigr):=\chi_{m+2k}(\kappa_{\theta_1})\,\frac{\tanh^k(t)}{\cosh^m(t)}\,\chi_{m}(\kappa_{\theta_2}),\qquad\theta_1,\theta_2\in\bbR,\ t\in\bbR_{\geq0}. \]
	Let $ N\in\bbZ_{>0} $, and let $ \Gamma $ be a subgroup of $ \Gamma(N) $ such that $ \#(\Gamma\cap\SO_2(\bbR))\mid m+2k $. If $ N $, $ m $, and $ k $ satisfy \eqref{eq:039}, then $ P_\Gamma F_{k,m}\in C^\infty\left(\Gamma\backslash\SL2(\bbR)\right) $ is not identically $ 0 $.
\end{Prop}
 
To better understand implications of Propositions \ref{prop:045} and \ref{prop:052}, we gather some properties of $ \M $ in the following lemma.

\begin{Lem}\label{lem:046}
	Let $ a,b\in\bbR_{>0} $. We have the following:
	\begin{enumerate}
		\item $ \M(\spacedcdot,b):\bbR_{>0}\to\left]0,1\right[ $ is strictly increasing.\label{lem:046:1}
		\item $ \M(a,\spacedcdot):\bbR_{>0}\to\left]0,1\right[ $ is strictly decreasing.\label{lem:046:2}
		\item $ \M(1,b)=1-2^{-\frac1b} $, $ \M(a,1)=2^{-\frac 1a} $, and $ \M(a,a)=\frac12 $.\label{lem:046:3}
		\item If $ 1<a<b $, then $ \frac{a-1}{a+b-2}<\M(a,b)<\frac a{a+b} $.\label{lem:046:4}
		\item If $ 1<b<a $, then $ \frac a{a+b}<\M(a,b)<\frac{a-1}{a+b-2} $.\label{lem:046:5}
	\end{enumerate}
\end{Lem}

\begin{proof}
	\eqref{lem:046:2} follows from \eqref{lem:046:1} by using the elementary identity $ \M(a,b)=1-\M(b,a) $. \eqref{lem:046:3} is elementary. \eqref{lem:046:4} and \eqref{lem:046:5} are the well-known mean-median-mode inequalities for the beta distribution; for their elementary proof, see \cite{payton}. 
	
	Let us prove \eqref{lem:046:1}. Using monotonicity of the Lebesgue integral, we have, for all $ a,\varepsilon\in\bbR_{>0} $,
	\begin{multline*}
	\int_0^{\M(a,b)}x^{a+\varepsilon-1}(1-x)^{b-1}\,dx<\M(a,b)^{\varepsilon}\int_0^{\M(a,b)}x^{a-1}(1-x)^{b-1}\,dx\\
	\overset{\eqref{eq:047}}=\M(a,b)^\varepsilon\int_{\M(a,b)}^1x^{a-1}(1-x)^{b-1}\,dx
	<\int_{\M(a,b)}^1x^{a+\varepsilon-1}(1-x)^{b-1}\,dx,
	\end{multline*}
	which implies that $ \M(a,b)<\M(a+\varepsilon,b) $. Hence, $ \M(\spacedcdot,b) $ is strictly increasing. 
\end{proof}

We note that the condition $ N>N_{k,m} $ in Propositions \ref{prop:045} and \ref{prop:052} is equivalent to $ N\geq\lfloor N_{k,m}\rfloor+1 $. 
One can easily calculate the exact value $ \lfloor N_{k,m}\rfloor+1 $, for a given pair $ (k,m)\in\bbZ_{\geq0}\times\left(\frac52+\frac12\bbZ_{\geq0}\right) $, using mathematical software (e.g., in R 3.3.2, $ \M(a,b) $ is implemented as \texttt{qbeta(0.5,a,b)}). 

In the following lemma, we give a close upper estimate of $ \lfloor N_{k,m}\rfloor+1 $ by elementary functions, for $ k\leq1000 $ and $ m>4 $. Our motivation is the idea of \cite{kerman} to approximate $ \M(a,b) $, for all $ a,b\in\bbR_{>1} $, by $ \M_\alpha(a,b):=\frac{a-\alpha}{a+b-2\alpha} $ for some $ \alpha\in\left]0,1\right[ $. By replacing $ \M $ by $ \M_{0.3131} $ in the definition of $ N_{k,m} $, we obtain the following:

\begin{Lem}
	Let $k\in\bbZ_{\geq0}$ and $m\in\frac92+\frac12\bbZ_{\geq0}$. We define $ C:=1.3738$ and
	\begin{equation}\label{eq:056}
	N_{k,m}^{close}:=4\sqrt{\frac{k+C}{m-4+C}\left(1+\frac{k+C}{m-4+C}\right)}.
	\end{equation}
	\begin{enumerate}	
		\item If $k\leq1000$, then $ \left\lceil N_{k,m}^{close}+6.204\right\rceil\in\left\{\left\lfloor N_{k,m}\right\rfloor+1,\ldots,\left\lfloor N_{k,m}\right\rfloor+8\right\} $.\label{lem:050:1}
		\item If $k\leq158$, then $\left\lceil N_{k,m}^{close}+0.8018\right\rceil\in\left\{\left\lfloor N_{k,m}\right\rfloor+1,\left\lfloor N_{k,m}\right\rfloor+2\right\} $.\label{lem:050:2}
		\item If $ k\leq1000 $ and $ m\geq 26.4+16.9431k $, then $ \left\lfloor N_{k,m}\right\rfloor+1=1 $.\\If $ k\leq1000 $ and $ m\leq25.34+16.9431k $, then $ \left\lfloor N_{k,m}\right\rfloor+1>1 $.\label{lem:050:3}
	\end{enumerate}
\end{Lem}

\begin{proof}
	A calculation in R \cite{rcalc} shows that \eqref{lem:050:1} is true if $ m<16970 $. Since both $ N_{k,m} $ and $ N_{k,m}^{close} $ are increasing in $ k $ and decreasing in $ m $ (for $ N_{k,m} $, this follows from Lemma \ref{lem:046}.\eqref{lem:046:1}--\eqref{lem:046:2}), we have, for $ m\geq16970 $ and $ k\leq1000 $,
	\begin{align*}
	1\leq\left\lfloor N_{k,m}\right\rfloor+1\leq\left\lfloor N_{1000,16970}\right\rfloor+1\overset{\text{R}}=1\quad\Rightarrow\quad&\left\lfloor N_{k,m}\right\rfloor+1=1,\\
	7\leq\left\lceil N_{k,m}^{close}+6.204\right\rceil\leq \left\lceil N^{close}_{1000,16970}+6.204\right\rceil\overset{\text{R}}=8\quad\Rightarrow\quad&\left\lceil N_{k,m}^{close}+6.204\right\rceil\in\left\{7,8\right\},
	\end{align*}
	which finishes the proof of \eqref{lem:050:1}. \eqref{lem:050:2} is proved analogously: one uses R to verify that it holds for $ m\leq2702 $ and to calculate $ \left\lfloor N_{158,2702.5}\right\rfloor+1=1 $ and $ \left\lceil N^{close}_{158,2702.5}+0.8018\right\rceil=2 $.
	
	Since $ N_{k,m} $ is decreasing in $ m $, the proof of \eqref{lem:050:3} comes down to checking in R that, for every $ k\in\{0,1,2,\ldots,1000\} $, we have 
	\begin{align*}
	\left\lfloor N_{k,\,m(k)}\right\rfloor+1=1,&\quad\text{where}\quad m(k):=\min\left(\left(\frac92+\frac12\bbZ_{\geq0}\right)\cap\bbR_{\geq26.4+16.9431k}\right),\\
	\left\lfloor N_{k,\,m'(k)}\right\rfloor+1>1,&\quad\text{where}\quad m'(k):=\max\left(\left(\frac92+\frac12\bbZ_{\geq0}\right)\cap\bbR_{\leq25.34+16.9431k}\right).\qedhere
	\end{align*} 
\end{proof}

\begin{Cor}\label{cor:054}
	Let $ N\in\bbZ_{>0} $, $ k\in\bbZ_{\geq0} $, $ m\in\frac92+\frac12\bbZ_{\geq0} $, and define $ N_{k,m}^{close} $ by \eqref{eq:056}. Let $ \Gamma $ be
	\[ \begin{cases}
	\text{a subgroup of }\bfGamma(N)\text{ such that }\Gamma\cap K=\{1\},&\text{ if }m\in\frac92+\bbZ_{\geq0},\\
	\text{a subgroup of }\Gamma(N)\text{ such that }\#(\Gamma\cap\SO_2(\bbR))\mid m+2k,&\text{ if }m\in\bbZ_{\geq5}.\\
	\end{cases} \]
	Then, $ P_\Gamma F_{k,m} $ is not identically $ 0 $ if one of the following holds:
	\begin{enumerate}
		\item $ k\leq 1000 $ and $ m\geq 26.4+16.9431k $
		\item $ k\leq 158 $ and $ N\geq\left\lceil N_{k,m}^{close}+0.8018\right\rceil $		
		\item $ k\leq 1000 $ and $ N\geq\left\lceil N_{k,m}^{close}+6.204\right\rceil $.
	\end{enumerate}
\end{Cor}



Next, we obtain a readable (if, in many cases covered by Corollary \ref{cor:054}, weaker) result for general $ k\in\bbZ_{\geq0} $. First, we apply Lemma \ref{lem:046}.\eqref{lem:046:3} to obtain
\begin{equation}\label{eq:051}
N_{0,m}=4\cdot2^{\frac1{m-2}}\sqrt{4^{\frac1{m-2}}-1},\qquad N_{k,4}=\frac4{2^{\frac1{k+2}}-2^{-\frac1{k+2}}},\qquad N_{k,k+4}=4\sqrt2 
\end{equation}
for all $ m\in\frac52+\frac12\bbZ_{\geq0} $ and $ k\in\bbZ_{\geq0} $. Next, the upper bounds on $ \M $ from Lemma \ref{lem:046}.\eqref{lem:046:4}--\eqref{lem:046:5} imply the following inequalities:
\begin{align}\label{eq:052}
	\begin{split}
	N_{k,m}<4\sqrt{\frac{k+2}{m-2}\left(1+\frac{k+2}{m-2}\right)},&\qquad\text{if }0<k<m-4,\\
	N_{k,m}<4\sqrt{\frac{k}{m-4}\left(1+\frac{k}{m-4}\right)},&\qquad\text{if }0<m-4<k.
	\end{split}
\end{align}
\eqref{eq:051} and \eqref{eq:052} imply the following corollary of Propositions \ref{prop:045} and \ref{prop:052}:

\begin{Cor}\label{cor:055}
	Let $ N\in\bbZ_{>0} $, $ k\in\bbZ_{\geq0} $, and $ m\in\frac52+\frac12\bbZ_{\geq0} $. Let $ \Gamma $ be
	\[ \begin{cases}
	\text{a subgroup of }\bfGamma(N)\text{ such that }\Gamma\cap K=\{1\},&\text{ if }m\in\frac52+\bbZ_{\geq0},\\
	\text{a subgroup of }\Gamma(N)\text{ such that }\#(\Gamma\cap\SO_2(\bbR))\mid m+2k,&\text{ if }m\in\bbZ_{\geq3}.\\
	\end{cases} \]
	Then, $ P_\Gamma F_{k,m} $ is not identically $ 0 $ if one of the following holds:
	\begin{enumerate}
		\item $ k=0 $ and $ N>4\cdot2^{\frac1{m-2}}\sqrt{4^{\frac1{m-2}}-1} $
		\item $ m=4 $ and $ {\displaystyle N>\frac4{2^{\frac1{k+2}}-2^{-\frac1{k+2}}} } $
		\item $ 0<k\leq m-4 $ and $ N\geq4\sqrt{\frac{k+2}{m-2}\left(1+\frac{k+2}{m-2}\right)} $
		\item $ 0<m-4\leq k $ and $ N\geq4\sqrt{\frac{k}{m-4}\left(1+\frac{k}{m-4}\right)} $.
	\end{enumerate}
\end{Cor}

\end{document}